\documentclass[a4paper,twoside,10pt]{amsart}
\usepackage{amsmath,amsfonts,amssymb,latexsym,textcomp, mathrsfs, amscd}
\usepackage{verbatim}

\usepackage[pdftex]{color, graphicx}

\newcommand{\qbinom}[3]{\genfrac[]{0pt}{}{#1}{#2} _{#3}}
\renewcommand{\P}{\mathbb P}
\newcommand{\C}{\mathbb C}
\newcommand{\R}{\mathbb R}
\newcommand{\N}{\mathbb N}
\newcommand{\Z}{\mathbb Z}
\newcommand{\dd}{\mathrm{d}}
\newcommand\Oh{\mathrm{O}}
\newcommand\oh{\mathrm{o}}

\newtheorem{bsp}{Example}[section]
\newtheorem{satz}[bsp]{Theorem}
\newtheorem{lem}[bsp]{Lemma}
\newtheorem{bem}[bsp]{Remark}
\newtheorem{toyo}[bsp]{Corollary}

\title{Convergence Properties of Kemp's $q$-Binomial Distribution}

\author{Stefan Gerhold}
\author{Martin Zeiner}

\thanks{S.~Gerhold received financial support by Christian Doppler Laboratory for Portfolio Risk Management (PRisMa Lab).
The fruitful collaboration and support by the Bank Austria Creditanstalt (BA-CA) and the Austrian Federal Financing
Agency (\"OBFA) through CDG is gratefully acknowledged.}
\address[Stefan Gerhold]{Vienna University of Technology, Wiedner Hauptstra\ss{}e 8--10,
1040 Vienna, Austria}

\thanks{M.~Zeiner was supported by the NAWI-Graz project and the Austrian Science Fund project S9611 of the National Research Network S9600
Analytic Combinatorics and Probabilistic Number Theory.}
\address[Martin Zeiner]{Graz University of Technology, Steyrergasse 30, 8010 Graz, Austria}

\email{sgerhold at fam.tuwien.ac.at}
\email{zeiner at finanz.math.tu-graz.ac.at}

\date{\today}

\begin{document}

\begin{abstract}
  We consider Kemp's $q$-analogue of the binomial distribution. Several convergence results
  involving the classical binomial, the Heine, the discrete normal, and the Poisson distribution are established.
  Some of them are $q$-analogues of classical convergence properties.
  Besides elementary estimates, we apply Mellin transform asymptotics.
\end{abstract}



\maketitle

\section{Introduction}

Kemp~\cite{KempKemp91} introduced the following $q$-analogue $KB(n,\theta ,q)$ of the binomial distribution:
\[
 \P (X_{KB} = x) = \qbinom{n}{x}{q}\frac{\theta^x q^{x(x-1)/2}}{(-\theta ,q)_n}, \qquad 0 \leq x \leq n,\ 0 < \theta,
\]
where 
\[
 \qbinom{n}{k}{q} = \frac{(q,q)_n}{(q,q)_k(q,q)_{n-k}} \qquad \textnormal{and} \qquad (z,q)_n = \prod_{i=0}^{n-1} (1-zq^i)
\]
are the $q$-binomial coefficient and the $q$-shifted factorial.
See~\cite{JKK05,Kemp02,KempNew90} for properties and applications of this distribution;
for an introduction to the $q$-calculus see~\cite{GasRah}. For $q\to 1$, this distribution tends to a binomial distribution:
\[
 KB(n,\theta ,q) \to B\left( n,\frac{\theta}{1+\theta} \right) .
\]
For $n \to \infty$, it tends to the Heine distribution $H(\theta )$:
\[
 \P (X_H = x) = \frac{q^{x(x-1)/2}\theta^x}{(q,q)_x} e_q(-\theta), \qquad x \geq 0,
\]
where
\[
 e_q (z) = \frac{1}{(z,q)_{\infty}}, \qquad z \in \C \setminus \{ q^{-i} : \ i = 1,2,\dots \}
\]
is a $q$-analogue of the exponential function, since $e_q((1-q)z) \to e^z$. The Heine distribution is a $q$-analogue of the Poisson distribution, since $H((1-q)\theta)\to P(\theta)$. Moreover, we need a second $q$-analogue of the exponential function, the function $E_q(z) = (-z,q)_{\infty}$.
Note that we have $e_q(z)E_q(-z) = 1$.

The random variable $X_{KB}$ can be written as the sum of independent Bernoulli random variables~\cite{KempNew90},
which leads to the expressions
\begin{equation}\label{eq:mean var}
 \mu = \sum_{i=0}^{n-1} \frac{\theta q^i}{1+\theta q^i} \qquad \textnormal{and}
  \qquad \sigma ^2=\sum_{i=0}^{n-1} \frac{\theta q^i}{(1+\theta q^i)^2}
\end{equation}
for the mean and variance.
We are now interested in sequences of random variables $X_n$ with $X_n \sim KB(n, \theta_n, q)$, in particular we show that there are analogues to the convergence of the classical binomial distribution to the Poisson distribution and the Normal distribution, and that the limits $q\to 1$ and $n \to \infty$ can be exchanged.
Section~\ref{sec:konvPar} deals with two cases of convergent parameter~$\theta_n$, in particular with the case of constant mean. In Section~\ref{sec:langPar} we show that, if~$\theta_n$ grows sub-exponentially, the normalized~$X_n$ converge to a discrete normal distribution. In Section~\ref{sec:schnellPar} we examine the case of an exponentially growing parameter sequence~$\theta_n$.

\section{Convergent Parameter}\label{sec:konvPar}

As noted above we consider sequences of random variables~$X_n$ with $X_n \sim X_{KB}(n, \theta_n(q),q)$.
In the present section we will provide convergence results for two different sequences~$\theta_n(q)$ which both tend to
a limit as~$n\to\infty$.
In the following we need
\begin{lem}\label{lem:prod}
Let $(\theta_n)$ be a sequence of real numbers with limit $\theta\geq0$. Then
\[
\lim_{n\to \infty} \prod_{i=0}^{n-1} \left( 1+ \theta_n q^i \right) = E_q(\theta ) .
\]
\end{lem}
\begin{proof}
For small $\epsilon > 0$ and $n$ large enough, we have
\[
  \prod_{i=0}^{n-1} \left( 1+ \left( \theta-\epsilon \right)q^i \right) \leq \prod_{i=0}^{n-1} \left( 1+\theta _n q^i \right)
   \leq \prod_{i=0}^{n-1} \left( 1+ \left( \theta+\epsilon \right)q^i \right),
\]
hence
\begin{align*}
  E_q(\theta-\epsilon) &= \lim_{n\to\infty}\prod_{i=0}^{n-1} \left( 1+ \left( \theta-\epsilon \right)q^i \right) \leq \liminf_{n\to\infty}\prod_{i=0}^{n-1} \left( 1+\theta _n q^i \right) \\
   &\leq \limsup_{n\to\infty}\prod_{i=0}^{n-1} \left( 1+\theta _n q^i \right) \leq\lim_{n\to\infty}\prod_{i=0}^{n-1} \left( 1+ \left( \theta+\epsilon \right)q^i \right) \\
   &= E_q(\theta+\epsilon).
\end{align*}
By continuity of~$E_q$, the lemma follows.
\end{proof}

The $q$-number $[x]_q$ is defined as
\[
 [x]_q := \frac{1-q^x}{1-q};
\]
for $q\to 1$, we have $[x]_q \to x$.
Now we can establish our first convergence result.
\begin{satz}
\[
\begin{CD}
  X_{KB}(n, \theta_n (q), q)  @>{n\to\infty}>>    H( (1-q)\lambda) \\
 @V{q\to1}VV  @VV{q \to 1}V \\
            B\left(n,\frac{\lambda}{n}\right)         @>>{n\to\infty}>    P(\lambda)
\end{CD}
\]
with $\theta_n(q) = \lambda/[n-\lambda]_q$.
\end{satz}
\begin{proof}
We only have to show $ X_{KB}(n, \theta_n (q), q) \to H((1-q)\lambda)$. Note that
\begin{eqnarray*}
\P (X_n = x) & = & \qbinom{n}{x}{q} \left( \frac{\lambda}{[n-\lambda ]_q}\right)^x \frac{q^{x(x-1)/2}}{\prod_{i=0}^{n-1} \left( 1+\frac{\lambda}{[n-\lambda]_q}q^i\right)} \\
             & = & \qbinom{n}{x}{q} \frac{\lambda^x (1-q)^x}{(1-q^{n-\lambda})^x} \frac{q^{x(x-1)/2}}{\prod_{i=0}^{n-1} \left( 1+\frac{\lambda(1-q)}{1-q^{n-\lambda}}q^i\right)} \\
             &\to& \frac{q^{x(x-1)/2}((1-q)\lambda)^x}{(q,q)_x} e_q(-(1-q)\lambda), 
\end{eqnarray*}
where the last line follows from Lemma~\ref{lem:prod}.
\end{proof}

For our next result, we note the following elementary fact.
\begin{lem}\label{lem:zp}
Let $f_n(x)$, $n \in \N$, be a sequence of functions that are increasing in~$x$, and suppose that for each~$n$ there is a unique solution~$x_n$ of $f_n(x)=0$. Moreover, assume that $f_n(x) < f_m(x)$ for $n < m$ and that~$f_n$ converges pointwise to a limit~$f$ with a unique solution $\hat{x}$ of $f(x)=0$. Then $(x_n)_{n \in \N}$ converges to $\hat{x}$.
\end{lem}
\begin{proof}
Since $f_n(x) < f_m(x)$ for $n < m$, the sequence $(x_n)$ is decreasing and $x_n \geq \hat{x}$. Therefore $x_n$ converges to a limit $\tilde{x}\geq\hat{x}$. Moreover,
\[
 0 = \lim_{n \to \infty} f_n(x_n) \geq \lim_{n \to \infty} f_n(\tilde{x}) = f(\tilde{x}),
\]
thus $\tilde{x}\leq\hat{x}$.
\end{proof}

Our second convergence result is analogous to the classical convergence
of the binomial distribution with constant mean to the Poisson distribution.

\begin{satz} Fix $\mu > 0$. Then we have
\[
\begin{CD}
  X_{KB}(n, \theta_n (q), q)  @>{n\to\infty}>>    H( \theta(q)) \\
 @V{q\to1}VV  @VV{q \to 1}V \\
            B\left(n,\frac{\mu}{n}\right)         @>>{n\to\infty}>    P(\mu)
\end{CD}
\]
with $\theta_n(q)$ such that $\mu_n=\mu$ and $\theta(q) = \lim_{n\to\infty} \theta_n(q)$.
\end{satz}
\begin{proof}
First we check that for given $\mu > 0$ and fixed $q$ there is a unique sequence $(\theta_n(q))_{n \geq N}$, such that $\mu_n(\theta_n(q),q)=\mu$. The function $\mu_n(\theta,q)$ is strictly increasing in $\theta$ and $\mu_n(0,q) = 0$.
Since
\[
 \mu_n(q^{-n+1},q) \geq \sum_{i=0}^{n-1} \frac{q^{i-n+1}}{2q^{i-n+1}} = \frac{n}{2}
\]
and $\mu_n(\theta,q)$ is continuous in $\theta$, there exists a unique solution $\theta_n(q)$ of $\mu_n(\theta,q) = \mu$ for each $n \geq2\mu$. As $\mu_n(\theta,q)$ is increasing in $n$, we can apply Lemma~\ref{lem:zp} to obtain $\lim_{n\to\infty} \theta_n = \theta(q)$, with $\theta(q)$ the unique solution of $\mu_{\infty}(\theta,q) = \mu$.
Thus $X_{KB}(n,\theta_n(q),q)\to H(\theta(q))$ by Lemma~\ref{lem:prod}.

The function $\mu_n(\theta,q)$ is also increasing in $q$, so we get $\theta_n(q) \to \tfrac{\mu}{n-\mu}$ (or equivalently $\tfrac{\theta_n (q)}{1+\theta_n (q)} \to \tfrac{\mu}{n}$) for $q\to 1$ by Lemma~\ref{lem:zp}. So $X_{KB}(n,\theta_n(q),q)\to B\left(n,\frac{\mu}{n}\right)$.

It remains to check that $\theta(q)/(1-q)$ converges to $\mu$ for $q\to 1$ (then $H(\theta(q)) \to P(\mu)$). The value $\theta(q)/(1-q)$ is the unique solution of $\mu_\infty((1-q)\theta,q)=\mu$. Moreover, $\mu_\infty((1-q)\theta,q)$ is increasing in $x$ and $\theta$ and $\lim_{q\to 1} \mu_\infty((1-q)\theta,q) = \theta$ (because $H((1-q)\theta) \to P(\theta)$). Thus we can again apply Lemma~\ref{lem:zp}.
\end{proof}

\section{Sub-Exponentially Increasing Parameter}\label{sec:langPar}

Now we consider parameter sequences $\theta_n = q^{-f(n)}$ with $f(n) \to \infty$ and $n-f(n) \to \infty$ for $n \to \infty$.
These assumptions on~$f(n)$ will be in force throughout the section.
Theorems~\ref{satz:konstGer0} and \ref{satz:konstWag0} and Lemmas~\ref{toyo:konst0}--\ref{toyo:mufloor}
are devoted to the asymptotic behavior of the sequence~$(\mu_n)$ of means.
As they tend to infinity, we will normalize our sequence of random variables
to~$(X_n-\mu_n)/\sigma_n$. Still, this sequence does not converge in distribution without
further assumptions on~$f(n)$. A fruitful way to proceed is to pick subsequences along
which the fractional part~$\{f(n)\}$ is constant. Theorem~\ref{th:Hauptsatz0} shows that this
induces convergence to discrete normal distributions.

To investigate the sequence of means,
we begin by providing an elementary estimate for the variance.
\begin{lem}\label{lem:sigmaKonv0}
If $\theta_n = q^{-f(n)}$ with the above assumptions on~$f(n)$, then the sequence of variances satisfies
$\sigma_n^2 \leq 2/(1-q)$.
\end{lem}
\begin{proof}
By~\eqref{eq:mean var}, the variance~$\sigma_n^2$ equals
\begin{align}
\sum_{i=0}^n \frac{q^{i-f(n)}}{(1+q^{i-f(n)})^2} &= \sum_{i=0}^{\lfloor f(n) \rfloor} \frac{q^{i-f(n)}}{(1+q^{i-f(n)})^2} + \sum_{i=\lfloor f(n) \rfloor+1}^n \frac{q^{i-f(n)}}{(1+q^{i-f(n)})^2} \nonumber\\
& =  \sum_{i=0}^{\lfloor f(n) \rfloor} \frac{q^{-\{f(n)\}-i}}{(1+q^{-\{ f(n) \} - i})^2} + \sum_{i=0}^{n-\lfloor f(n) \rfloor-1} \frac{q^{i+1-\{ f(n) \}}}{(1+q^{i+1-\{ f(n) \}})^2}  \label{equ:lemSigmaKonv0}\\
& <  \sum_{i=0}^{\infty} \frac{1}{q^{-\{ f(n) \} - i}} + \sum_{i=0}^{\infty} q^{i+1-\{ f(n)\}} \nonumber\\
& \leq  \sum_{i=0}^\infty q^i + \sum_{i=0}^\infty q^i = \frac{2}{1-q}. \nonumber
\end{align}
\end{proof}

The following result about the sequence of means does not reveal the structure of the $\mathrm{O}(1)$ term, but
will be useful later on (Lemma~\ref{toyo:konstFloor0}).
\begin{satz}\label{satz:konstGer0}
Let $X_n \sim KB(n,\theta_n,q)$ with $\theta_n = q^{-f(n)}$. Then, for $n\to\infty$,
\[
 \mu_n = f(n) + c( \{ f(n) \},q) + \mathrm{o}(1),
\]
where
\[
 c( \{ f(n) \},q) := 1 - \frac{1}{1+q^{-\{ f(n) \}}} - \left\{ f(n) \right \} -\sum_{\ell \geq 0} \frac{1}{1+q^{-\ell-\{ f(n) \}-1}} + \sum_{\ell \geq 0}\frac{1}{1+q^{-\ell+\{ f(n) \}-1}} = \mathrm{O}(1).
\]
\end{satz}
\begin{proof} We start {}from
\begin{equation}\label{eq:mu}
 \mu_n = \sum_{i=0}^{n-1} \frac{q^{i-f(n)}}{1+q^{i-f(n)}} = \sum_{i=0}^{n-1} \frac{1}{1+q^{f(n)-i}}
\end{equation}
and split the sum into two parts (w.l.o.g.\ $f(n)<n$):
\begin{eqnarray*}
\sum_{i=0}^{\lfloor f(n) \rfloor - 1} \frac{1}{1+q^{f(n)-i}} & = & \sum_{i=0}^{\lfloor f(n) \rfloor - 1} \sum_{\ell \geq 0}(-1)^\ell q^{\ell (f(n)-i)} \\
& = & \sum_{\ell \geq 0} (-1)^\ell q^{\ell f(n)} \sum_{i=0}^{\lfloor f(n) \rfloor - 1} q^{-\ell i} \\
& = & \lfloor f(n) \rfloor + \sum_{\ell \geq 1} (-1)^\ell q^{\ell f(n)}\frac{1-q^{-\ell\lfloor f(n) \rfloor}}{1-q^{-\ell}} \\
& = & \lfloor f(n) \rfloor - \sum_{\ell \geq 1} \frac{(-1)^\ell q^{\ell \{ f(n) \}}}{1-q^{-\ell}} + \Oh \left( q^{f(n)} \right) \\
& = & \lfloor f(n) \rfloor + \sum_{\ell \geq 1} \frac{q^\ell(-1)^\ell q^{\ell \{ f(n) \}}}{1-q^{\ell}} + \Oh \left( q^{f(n)} \right) \\
& = & \lfloor f(n) \rfloor + \sum_{\ell \geq 1} q^\ell(-1)^\ell q^{\ell \{ f(n) \}} \sum_{j\geq 0} q^{\ell j} + \Oh \left( q^{f(n)} \right) \\
& = & \lfloor f(n) \rfloor + \sum_{j \geq 0} \sum_{\ell \geq 1} \left( -q^{j+1+\{ f(n) \}} \right)^\ell + \Oh \left( q^{f(n)} \right) \\
& = & \lfloor f(n) \rfloor + \sum_{j \geq 0} \frac{-q^{j+1+\{ f(n) \}}}{1+q^{j+1+\{ f(n) \}}} + \Oh \left( q^{f(n)} \right) \\
& = & \lfloor f(n) \rfloor - \sum_{j \geq 0} \frac{1}{1+q^{-j-1-\{ f(n) \}}} + \Oh \left( q^{f(n)} \right) \\
\end{eqnarray*}
For the upper portion of the sum, we find
\begin{align*}
\sum_{i= \lfloor f(n) \rfloor + 1}^{n-1} \frac{1}{1+q^{f(n) -i}} &= \sum_{i= \lfloor f(n) \rfloor + 1}^{\infty} \frac{1}{1+q^{f(n) -i}} + \Oh
 \left(q^{n-f(n)} \right) \\
 &= \sum_{i= 0}^{\infty} \frac{1}{1+q^{\{f(n)\} -i-1}} + \Oh\left(q^{n-f(n)} \right),
\end{align*}
since
\[
 \sum_{i=n}^{\infty} \frac{1}{1+q^{f(n) -i}} = \sum_{i=0}^\infty \frac{1}{1+q^{f(n)-n-i}} \leq \sum_{i=0}^\infty \frac{1}{q^{f(n)-n-i}} = q^{n-f(n)}\frac{1}{1-q} .
\]
\end{proof}
%

In the limit $q\to 1$, the term $c(\{f(n)\},q)$ tends to $\frac12$. To see this, apply the Euler-Maclaurin formula to
\[
f(x) = \frac{1}{1+q^{-x-b}}
\]
with $b > 0$, which yields
\begin{equation}\label{equ:eul}
\sum_{\ell \geq 0} f(\ell ) = \int\limits_{0}^{\infty} f(x) \dd x + \frac{f(0)}{2} +\frac{1}{12}f'(x) \big|_{x=0}^{\infty} + R_2
\end{equation}
with
\[
R_2 = -\frac12\int\limits_{0}^{\infty} B_2\left( \{x\} \right) f''(x) \dd x .
\]
Since
\[
f''(x) = \frac{ (\log q )^2q^{x+b}\left( 1-q^{x+b} \right)}{\left(1+q^{x+b}\right)^3}
\]
does not change sign, we have
\begin{align*}
| R_2 | & \leq \frac{1}{12} \int\limits_{0}^{\infty} |f''(x)| \dd x = \frac{1}{12}\int\limits_{0}^{\infty} f''(x) \dd x \\
 & = -(\log q)q^{-b}(1+q^{-b})^{-2} = \oh(1), \qquad q\to 1.
\end{align*}
The first integral in~(\ref{equ:eul}) is
\begin{align*}
\int\limits_{0}^{\infty} f(x) \dd & x = -\frac{\log \left( 1+q^b\right)}{\log q} = \frac{\log 2}{1-q} - \frac{\log 2 + b}{2} + \Oh\left( 1-q\right), \qquad q\to 1.
\end{align*}
So we have
\[
\sum_{\ell \geq 0} f(\ell) = \frac{\log 2}{1-q} - \frac{\log 2 + b}{2} +\frac{1}{4} + \oh(1), \qquad q \to 1.
\]
Application to the sums appearing in $c(\{f(n)\},q)$ gives
\[
c(\{f(n)\},q)= \frac12-\{f(n)\} + \{f(n)\} + \oh(1).
\]
Note that for $q \to 1$ the error term in the representation for $\mu_n$ increases. This is why the limits for $q \to 1$ and $n \to \infty$ can't be exchanged ($\mu_n$ tends to $n/2$ for $q \to 1$).

The following theorem provides a different representation of the $\mathrm{O}(1)$ term
{}from Theorem~\ref{satz:konstGer0}, which shows that it is a $\tfrac12$-periodic function of~$f(n)$.
\begin{satz}\label{satz:konstWag0}
Let $X_n \sim KB(n,\theta_n,q)$ with $\theta_n = q^{-f(n)}$. Then, as $n\to\infty$,
\begin{eqnarray}\label{equ:satzKonstWag0}
 \mu_n = f(n) + \frac12+\sum_{k > 0} \frac{2\pi\sin(2kf(n)\pi)}{\log q \sinh \left( \frac{2k\pi^2}{\log q} \right)} + \Oh \left( q^{\min(f(n)/2,n-f(n))} \right).
\end{eqnarray}
\end{satz}
\begin{proof}
We write
\[
 \mu_n = \sum_{i=0}^{n-1} \frac{1}{1+q^{f(n) - i }} = \sum_{i=0}^{\infty} \frac{1}{1+q^{f(n) - i}} + \Oh \left( q^{n-f(n)} \right)
\]
and apply the Mellin transformation~\cite{FlaGouDum95} to
\[
 h(t) = \sum_{i=0}^{\infty} \frac{1}{1+tq^{-i}}.
\]
By the linearity of the Mellin transformation~$\mathcal{M}$ and the properties $\mathcal{M} \left( \frac{1}{1+t} \right) = \frac{\pi}{\sin \pi s}$ and $\mathcal{M}h(\alpha t)(s)= \alpha^{-s}\mathcal{M}(h)(s)$, we see that
\[
 \mathcal{M}(h)(s) = \sum_{i=0}^{\infty} \left( q^{-i}\right) ^{-s} \frac{\pi}{\sin \pi s} = \frac{1}{1-q^s} \frac{\pi}{\sin \pi s}. 
\]
Exchanging~$\mathcal{M}$ and the sum is permitted by the monotone convergence theorem. {}From the inverse transformation formula we get
\begin{equation}
 h \left( q^{f(n)} \right) = \frac{1}{2 \pi i } \int\limits_{c-i\infty}^{c+i\infty} q^{-f(n)s} \frac{1}{1-q^s} \frac{\pi}{\sin \pi s} \dd s 
\end{equation}
for $c\in (0,1)$. To evaluate this integral, we choose the integration contour $\gamma_k = \gamma_{k,1} \cup \gamma_{k,2} \cup \gamma_{k,3} \cup \gamma_{k,4}$ with
\begin{eqnarray*}
\gamma_{k,1} & = & \left\{ s\ | \ s=\frac12 + iv\ : \ -T_k \leq v \leq T_k \right\}\\
\gamma_{k,2} & = & \left\{ s\ | \ s=u +iT_k:\ -\frac12 \leq u \leq \frac12 \right\}\\
\gamma_{k,3} & = & \left\{ s\ | \ s=-\frac12 + iv\ : \ -T_k \leq v \leq T_k  \right\}\\
\gamma_{k,4} & = & \left\{ s\ | \ s=u -iT_k:\ -\frac12 \leq u \leq \frac12 \right\}
\end{eqnarray*}
where $T_k=\frac{2\pi}{\log q}\left( k +\frac14 \right)$. Then
\[
 h \left( q^{f(n)} \right) = \lim_{k\to \infty} \frac{1}{2\pi i}\int\limits_{\gamma_{k,1}} = -\lim_{k\to \infty} ( \frac{1}{2\pi i}\int\limits_{\gamma_{k,2}}
+\frac{1}{2\pi i}\int\limits_{\gamma_{k,3}}+ \frac{1}{2\pi i}\int\limits_{\gamma_{k,4}}  + \sum \textnormal{residues} ),
\]
since the integral on the left side exists. Now we estimate the integrals on the right side.
\begin{eqnarray*}
\left| \int\limits_{\gamma_{k,3}} q^{-f(n)s} \frac{1}{1-q^s} \frac{\pi}{\sin \pi s} \dd s \right| & = & 
 \left| \int\limits_{-T_k}^{T_k} q^{-f(n) (-\frac12+iv)} \frac{1}{1-q^{-\frac12+iv}}\frac{\pi}{\sin(\pi(-\frac12+iv))} \dd v \right| \\
& \leq & \pi q^{\frac{f(n)}{2}} \int\limits_{-\infty}^{\infty} \frac{1}{\left| 1-q^{-\frac12+iv}\right|}\frac{1}{\left| \sin(\pi(-\frac12+iv)) \right|} \dd v \\
& \leq & \pi q^{\frac{f(n)}{2}} \frac{1}{1-q^{-\frac12}} \int\limits_{-\infty}^{\infty} \frac{1}{\sqrt{ \sin^2 \frac{\pi}{2} + \sinh^2 \pi v }} \dd v \\
& = & q^{\frac{f(n)}{2}} \frac{\pi}{1-q^{-\frac12}}
\end{eqnarray*}
\begin{eqnarray*}
\left| \int\limits_{\gamma_{k,2}} q^{-f(n)s} \frac{1}{1-q^s} \frac{\pi}{\sin \pi s} \dd s \right| & = &
 \left| \int\limits_{-\frac12}^{\frac12} q^{-f(n) (u+iT_k)} \frac{1}{1-q^{u+iT_k}}\frac{\pi}{\sin(\pi(u+iT_k))} \dd u \right| \\
& \leq & \pi \int\limits_{-\frac12}^{\frac12} q^{-f(n) u} \frac{1}{\left| 1-q^{u+iT_k} \right| }\frac{1}{\sqrt{\sin^2\pi u + \sinh^2 \pi T_k}} \dd u \\
& \leq & \pi q^{-\frac12 f(n) } \int\limits_{-\frac12}^{\frac12} \frac{1}{q^u\left| \sin (T_k\log q) \right|} \frac{1}{\sqrt{\sinh^2 \pi T_k}} \dd u \\
& \leq & \frac{\pi q^{-\frac12 f(n) }}{\sinh \pi T_k} \int\limits_{-\frac12}^{\frac12} \frac{1}{q^u} \dd u \stackrel{k \to \infty}{\longrightarrow} 0
\end{eqnarray*}
The integral over $\gamma_{k,4}$ is treated similarly.
Now let us compute the residues: $\frac{1}{1-q^s}$ has simple poles at $z_k:=\frac{2\pi i k}{\log q}$, and $\frac{1}{\sin \pi s}$ has a simple pole at~$0$. First we consider the residue at~$z_k$ for $k\neq 0$: 
\begin{eqnarray*}
\lim_{z\to z_k} (z-z_k)q^{-f(n)z}\frac{1}{1-q^z} \frac{\pi}{\sin \pi z} & = & q^{-f(n) \frac{2\pi i k}{\log q} }\frac{\pi}{\sin\left( \frac{2\pi i k}{\log q}\pi \right)}
\lim_{z \to z_k} \frac{z-z_k}{1-q^z} \\
& = & e^{-f(n) 2 \pi i k } \frac{\pi}{i\sinh\left( \frac{2\pi^2k}{\log q} \right)} \frac{1}{-\log q}
\end{eqnarray*}
The sum extended over the residues at the poles~$z_k$, $k\neq0$, therefore equals
\[
 \sum_{k \neq 0} \frac{i\pi e^{-2if(n)k\pi}}{\log q \sinh\left(\frac{2k\pi^2}{\log q} \right)} .
\]
Putting together the summands $k$ and $-k$,
\begin{eqnarray*}
 e^{-2if(n)k\pi} - e^{2if(n)k\pi} & = & \cos ( -2if(n)k\pi) + i \sin(-2f(n)k\pi) - \cos(2f(n)k\pi) - i\sin(2f(n)k\pi) \\
&=& -2i \sin(2f(n)k\pi),
\end{eqnarray*}
we obtain
\[
 \sum_{k > 0 }\frac{2\pi \sin(2kf(n) \pi)}{\log q \sinh\left(\frac{2k\pi^2}{\log q} \right)} .
\]
Finally, by the expansions
\begin{eqnarray*}
 q^{-f(n) s} & = & 1-f(n) \log q \ s + \Oh(s^2)  \\
 \frac{1}{1-q^s} & = & -\frac{1}{\log q \ s} + \frac12 + \Oh(s) \\
 \frac{\pi}{\sin \pi s} & = & \frac{1}{s} + \Oh(s),
\end{eqnarray*}
the residue at $z_0=0$ 
is $f(n) + \frac12$.
\end{proof}

It is worthwhile to evaluate the sum in~\eqref{equ:satzKonstWag0} in the limit~$q\to0$.
First note that, if~$n$ is fixed and~$q\to 0$, then~\eqref{eq:mu} easily yields
\[
 \mu _n \to 
 \begin{cases}
   f(n)+1-\{ f(n) \} & \text{if}\ \{ f(n) \} >0 \\
   f(n) + \tfrac12 & \text{if}\  \{ f(n) \} =0 
 \end{cases}.
\]
Moreover, the~$\Oh$-term in~\eqref{equ:satzKonstWag0} is~$\oh(1)$ for~$q\to0$,
as follows readily {}from the estimates in the proof of Theorem~\ref{satz:konstWag0}.
These two facts combined imply
\[
  \lim_{q\to0} \sum_{k > 0} \frac{2\pi\sin(2kf(n)\pi)}{\log q \sinh \left( \frac{2k\pi^2}{\log q} \right)}
  = \begin{cases}
   \tfrac12 - \{f(n)\} & \text{if}\ \{ f(n) \} >0 \\
   0 & \text{if}\  \{ f(n) \} =0 
 \end{cases}.
\]
Note that in the special case~$f(n)=\alpha n$ with positive~$\alpha$, the summands
tend to the summands of the Fourier series of~$\tfrac12 - \{f(n)\}$, if~$\{\alpha n\}>0$.

After this analysis of the means~$\mu_n$, we turn our attention to the
convergence of the distributions. As mentioned above, a sufficient
condition for convergence of $(X_n-\mu_n)/\sigma_n$ is that
$\{f(n)\}$ is constant.
\begin{lem}\label{toyo:konst0} If we choose a subsequence~$(n_k)$ such that $\{f(n_k)\} = \beta$ constant, then:
\begin{itemize}
 \item[(a)] For $k\to \infty$
\[
 \mu_{n_k} = f(n_k) + c(\beta,q) + \oh(1) , 
\]
with $c(\beta,q)$ is a constant depending on $\beta$ and $q$.
 \item[(b)] 
   \begin{itemize}
     \item[(i)] $c(0,q) = c(1/2,q) = 1/2$
     \item[(ii)] $c(\beta,q) + c(-\beta,q) = 1$
   \end{itemize}
\end{itemize}
\end{lem}
\begin{proof}
Use (\ref{equ:satzKonstWag0}) and simple properties of $\sin$.
\end{proof}

\begin{lem}\label{toyo:konstFloor0} Set $\beta= \{f(n)\}$. Then
\[
\lfloor c(\beta, q) + \beta \rfloor = \left\{ \begin{array}{ll} 0 & 0 \leq \beta < 1/2 \\ 1 & 1/2 \leq \beta < 1 \end{array} \right.
\]
\end{lem}
\begin{proof}
We define
\[
 \hat c ( \{ f(n) \},q) := c( \{ f(n) \},q) - 1 + \{ f(n) \}.
\]
 By Theorem~\ref{satz:konstGer0}, $\hat c (\beta, q)$ is strictly increasing in $\beta$. Therefore we have for $ 0 \leq \beta < 1/2$
\[
 \hat c (0,q) = -\frac12 \leq \hat c (\beta ,q) < \hat c (1/2, q) = - \frac{1}{1-q^{-1/2}} + \frac{1}{1-q^{-1/2}} = 0.
\]
Thus
\begin{eqnarray}\label{equ:toyoKonstFloor01}
 \frac12 - \beta \leq c (\beta ,q) < 1 - \beta \qquad \textnormal{and} \qquad \frac12 \leq c (\beta , q) + \beta < 1 .
\end{eqnarray}
Similarly, we get for $ 1/2 \leq \beta < 1$
\begin{eqnarray}\label{equ:toyoKonstFloor02}
 1-\beta \leq c(\beta, q) < \frac12 \qquad \textnormal{and} \qquad 1 \leq c(\beta, q)+\beta < \frac12 + \beta <\frac32.
\end{eqnarray}

\end{proof}

\begin{lem}\label{toyo:mufloor} 
\begin{itemize}
 \item[(i)] If $\beta \neq \frac12$, then $f(n) + c(\beta,q) \not \in \Z$. Thus 
\[
 \lfloor \mu_n \rfloor = \lfloor f(n) + c(\beta,q) \rfloor = \lfloor f(n) \rfloor + \lfloor \beta + c(\beta,q) \rfloor .
\]
\item[(ii)] For $\beta = \frac12$,
\[
\mu_n > f(n)+\frac12 \ \ \textnormal{  if } \ 2f(n) \leq n-1 \qquad \textnormal{and} \qquad \mu_n < f(n)+\frac12 \ \ \textnormal{   if  } \ 2f(n) \geq n .
\]
Thus
\[
\lfloor \mu_n \rfloor = f(n)+\frac12\ \ \textnormal{  if } \ 2f(n) \leq n-1 \qquad \textnormal{and} \qquad \lceil \mu_n \rceil = f(n)+\frac12 \ \ \textnormal{   if  } \ 2f(n) \geq n .
\]
\end{itemize}
\end{lem}
\begin{proof}
(i): {}From (\ref{equ:toyoKonstFloor01}) we get for $ 0 \leq \beta < 1/2$
\[
 f(n) + \frac12 - \beta < f(n) + c(\beta ,q) < f(n) + 1 -\beta
\]
and therefore
\[
 \lfloor f(n) \rfloor + \frac12 < f(n) + c(\beta, q) < \lfloor f(n) \rfloor + 1.
\]
Similarly, {}from (\ref{equ:toyoKonstFloor02}) we get for $ 1/2 < \beta < 1$
\[
 \lfloor f(n) \rfloor + 1 < f(n) + c(\beta ,q) < \lfloor f(n) \rfloor + \frac32.
\]
(ii): Assume $2f(n) \leq n-1$. Then
\begin{eqnarray*}
\sum_{i=0}^{n-1} \frac{q^{i-f(n)}}{1+q^{i-f(n)}} & = & \sum_{i=0}^{f(n) - \frac12} \frac{q^{i-f(n)}}{1+q^{i-f(n)}} +
  \sum_{i=f(n)+\frac12}^{2f(n)}\frac{q^{i-f(n)}}{1+q^{i-f(n)}} + \sum_{2f(n)+1}^{n-1}
   \frac{q^{i-f(n)}}{1+q^{i-f(n)}} \\
& = & \sum_{i=0}^{f(n)-\frac12} \frac{q^{-i-\frac12}}{1+q^{-i-\frac12}} + \sum_{i=0}^{f(n)-\frac12} \frac{q^{i+\frac12}}{1+q^{i+\frac12}} + \oh(1) \\
& = & f(n) + \frac12 + \oh(1).
\end{eqnarray*}
We used $\frac{q}{1+q} + \frac{q^{-1}}{1+q^{-1}} = 1$; the $\oh(1)$-term is non-negative (and vanishes only for $2f(n)=n-1$).
If $2f(n) \geq n$, then the third sum vanishes and the second sum just runs up to $n-1 < 2f(n)$, so $\mu_n < f(n) + \frac12$.
\end{proof}

Note that similarly to the proof of~(ii) we can prove the properties of $c(\beta,q)$ in Lemma~\ref{toyo:konst0}~(b). Especially one can directly show the following theorem for
$\beta = \frac12$ and $\beta = 0$.

\begin{satz}\label{th:Hauptsatz0}
 Let $(n_k)_{k\in \N}$ be an increasing sequence of natural numbers and $X_{n_k} \sim KB(n_k,\theta_{n_k},q)$ with $\theta_{n_k} = q^{-f(n_k)}$ and $\{ f(n_k) \} = \beta$ constant. Recall that we always assume~$f(n)\to\infty$ and $n-f(n)\to\infty$.
 Then $(X_{n_k}-\mu_{n_k})/\sigma_{n_k}$ converges for $k \to\infty$ to a limit $X$, with
\begin{equation}\label{equ:Hauptsatz01}
 \P \left( X = -\left(\beta + c \right)\frac{1}{\sigma} + \frac{1}{\sigma}x\right) = e_q(q)e_q(-q^{\beta})e_q(-q^{1-\beta})q^{(x-1)(x-2\beta)/2},
 \qquad x\in\Z,
\end{equation}
if $\beta < 1/2$,
\begin{equation}\label{equ:Hauptsatz02}
 \P \left( X = -\left(\beta + c -1 \right)\frac{1}{\sigma} + \frac{1}{\sigma}x\right) = e_q(q)e_q(-q^{\beta})e_q(-q^{1-\beta})q^{x((1+x)-2\beta)/2},
 \qquad x\in\Z,
\end{equation}
if $\beta > 1/2$ and
\begin{equation}\label{equ:Hauptsatz03}
 \P \left( X = \frac{1}{\sigma}x\right) = e_q(q)e_q(-q^{1/2})^2 q^{x^2/2}, \qquad x\in\Z,
\end{equation}
if $\beta=1/2$, where $c=c(\beta,q)$ is the constant {}from Lemma~\ref{toyo:konst0} and $\sigma = \lim_{k\to\infty} \sigma_{n_k}$.
The distribution of~$X$ is symmetric iff $\beta=0$ or $\beta=1/2$.
\end{satz}
\begin{proof}
For simplicity we write in the following $n$ instead of $n_k$. {}From (\ref{equ:lemSigmaKonv0}) one gets that the $\sigma_{n}$ converge.
Consider the case $\beta \neq 1/2$:
\begin{align}
 \P (X_n = \lfloor \mu_n \rfloor + x)& = \qbinom{n}{ \lfloor \mu_n \rfloor + x}{q} \frac{ \theta_n ^{\lfloor \mu_n \rfloor + x} q^{(\lfloor \mu_n \rfloor + x)(\lfloor \mu_n \rfloor + x-1)/2}}{\prod_{i=0}^{n-1} \left( 1+ \theta_nq^i\right)} \notag\\
&= \qbinom{n}{ \lfloor \mu_n \rfloor + x}{q} \frac{ q^{-(\lfloor \mu_n \rfloor + x)f(n)+(\lfloor \mu_n \rfloor + x)(\lfloor \mu_n \rfloor + x-1)/2}}{\prod_{i=0}^{n-1} \left( 1+ \frac{q^i}{q^{f(n)}}\right)}. \label{eq:expo}
\end{align}
The product in the denominator equals
\begin{align}
\prod_{i=0}^{n-1} \left( 1+ \frac{q^i}{q^{f(n)}}\right) & =  \prod_{i=0}^{\lfloor f(n) \rfloor} \left( 1+ \frac{q^i}{q^{f(n)}}\right) 
\prod_{i=\lfloor f(n) \rfloor +1}^{n-1} \left( 1+ \frac{q^i}{q^{f(n)}}\right) \notag \\
& =  q^{-f(n) \left( \lfloor f(n) \rfloor+1\right)+ \left( \lfloor f(n) \rfloor+1\right)\lfloor f(n) \rfloor /2}
      \prod_{i=0}^{\lfloor f(n) \rfloor} \left(q^{f(n)-i}+1\right) \notag \\
  &  \qquad    \times \prod_{i=0}^{n-\lfloor f(n) \rfloor -2}\left(1+q^{i+\lfloor f(n) \rfloor - f(n)+1} \right) \notag \\
& =  q^{-f(n) \left( \lfloor f(n) \rfloor+1\right)+ \left( \lfloor f(n) \rfloor+1\right)\lfloor f(n) \rfloor /2}
      \prod_{i=0}^{\lfloor f(n) \rfloor} \left(q^{f(n)-\lfloor f(n)\rfloor}q^{\lfloor f(n)\rfloor-i}+1\right)\notag  \\
&  \qquad \times\prod_{i=0}^{n-\lfloor f(n) \rfloor -2}\left(1+q^iq^{\lfloor f(n) \rfloor - f(n)+1} \right) \notag \\
& =  q^{-f(n) \left( \lfloor f(n) \rfloor+1\right)+ \left( \lfloor f(n) \rfloor+1\right)\lfloor f(n) \rfloor /2}
      \left( -q^{\beta};q\right)_{\lfloor f(n)\rfloor+1} \left( -q^{-\beta+1};q\right)_{n-\lfloor f(n) \rfloor-2}. \label{eq:prod}
\end{align}
The second equality uses the easy relation
\begin{equation}\label{equ:1}
  \prod_{i=0}^{n-1}(1+zq^i) = q^{n(n-1)/2}z^n \prod_{i=0}^{n-1} (1+(zq^i)^{-1}).
\end{equation}
The last two terms in~\eqref{eq:prod} tend to $e_q\left( -q^{\beta} \right)$ and $e_q\left( -q^{-\beta+1} \right)$.
The $q$-binomial coefficient in~\eqref{eq:expo} tends to~$e_q(q)$. The exponent of~$q$ resulting
{}from~\eqref{eq:expo} and~\eqref{eq:prod} is
\begin{align*}
 -(\lfloor \mu_n & \rfloor + x)f(n)+(\lfloor \mu_n \rfloor + x)(\lfloor \mu_n \rfloor + x-1)/2
    +f(n) \left( \lfloor f(n) \rfloor+1\right)- \left( \lfloor f(n) \rfloor+1\right)\lfloor f(n) \rfloor /2 \\
&= \ \left( \lfloor f(n) \rfloor + \lfloor \beta +c \rfloor + x \right) \left( \lfloor f(n) \rfloor + \lfloor \beta + c \rfloor - 1 + x \right)/2
-\left( \lfloor f(n) \rfloor + \lfloor \beta +c \rfloor + x \right)f(n) \\
&\qquad + f(n) \left(  \lfloor f(n) \rfloor+1\right)- \left( \lfloor f(n) \rfloor +1\right) \lfloor f(n) \rfloor /2 \\
&= \ \frac12 (x-1+\delta)(\delta-2f(n)+2\lfloor f(n) \rfloor + x) \\
&= \frac12 (x-1+\delta )(\delta-2\beta +x),
\end{align*}
where  $c=c(\beta,q)$ and 
\[
\delta = \lfloor \beta + c \rfloor = \left\{ \begin{array}{ll} 0 & \beta < 1/2 \\ 1 & \beta > 1/2 \end{array} \right.
\]
by Lemma~\ref{toyo:konstFloor0}.  Putting things together, we obtain
\[
 \P (X_n = \lfloor \mu_n \rfloor + x) \to e_q(q)e_q\left( -q^{\beta} \right) e_q\left( -q^{-\beta+1}\right) q^{\frac{(\delta + x -1)(\delta +x-2\beta)}{2}} .
\]
By normalizing $X_n$ we get (\ref{equ:Hauptsatz01}) and (\ref{equ:Hauptsatz02}). The distribution of~$X$ is symmetric iff
\begin{align*}
 - \left( \beta + c -\lfloor \beta + c \rfloor \right) &= - \left( \beta + c -\lfloor \beta + c \rfloor \right) + 1  \\
 \Longleftrightarrow \qquad \beta + c -\lfloor \beta + c \rfloor &= \frac12.
\end{align*}
This is true for $\beta=0$ by Lemma~\ref{toyo:konst0}~(b)~(i). For $0 < \beta < \frac{1}{2}$ we have $\lfloor \beta + c \rfloor = 0$ by Lemma~\ref{toyo:konstFloor0}. But then we must have $\beta + c = \frac12$, which would contradict~\eqref{equ:toyoKonstFloor01} (since equality only holds for $\beta=0$). For $\beta>\frac{1}{2}$ we must have $\beta + c = \frac{3}{2}$ by Lemma~\ref{toyo:konstFloor0}, but this would be a contradiction to (\ref{equ:toyoKonstFloor02}).

For $\beta=1/2$ define
\[
 H(\mu_n) := \left\{ \begin{array}{ll} \lfloor \mu_n \rfloor & \textnormal{if }\  2f(n) \leq n-1 \\ \lceil \mu_n \rceil & \textnormal{if }\ 2f(n) \geq n \end{array} \right.
\]
Then
\[
\P \left( X_n = H(\mu_n) +x \right) =\qbinom{n}{ H(\mu_n) + x}{q} \frac{ q^{-(H(\mu_n) + x)f(n)+(H(\mu_n) + x)(H(\mu_n) + x-1)/2}}{\prod_{i=0}^{n-1} \left( 1+ \frac{q^i}{q^{f(n)}}\right)}.
\]
The $q$-binomial-coefficient tends to~$e_q(q)$, and the product can be transformed as above.
This time the exponent of~$q$ equals
\begin{align*}
 -(H(\mu_n) &+ x)f(n)+(H(\mu_n) + x)(H(\mu_n) + x-1)/2 +f(n) \left( \lfloor f(n) \rfloor+1\right) \\
 &\qquad - \left( \lfloor f(n) \rfloor+1\right)\lfloor f(n) \rfloor /2 \\
& = -(f(n)+\frac12 + x)f(n)+(f(n)+\frac12 + x)(f(n)-\frac12 + x)/2 \\
&\qquad  +f(n) \left( f(n)-\frac12 +1\right)- \left( f(n)-\frac12+1\right)\left( f(n)-\frac12 \right) /2 \\
& = \frac{x^2}{2}.
\end{align*}
So we have
\[
 \P \left( X_n = H(\mu_n) +x \right) \to e_q(q)e_q\left( -q^{\frac{1}{2}}\right)^2q^{\frac{x^2}{2}} .
\]
By normalizing $X_n$ we get (\ref{equ:Hauptsatz03}).
\end{proof}

The discrete normal distribution is defined by
\[
 \P (X=x) = \frac{q^{-x\alpha}q^{x^2/2}}{\sum_{k=-\infty}^{\infty} q^{-k\alpha}q^{k^2/2}} \qquad \alpha \in \R, \qquad x\in \Z .
\]
So the limit distributions in the preceding theorem are normalized discrete normal distributions with parameters 
\[
\begin{array}{ll} 
\alpha = \frac12 +\beta & \textnormal{if } \beta < \frac{1}{2} \\
\alpha = -\frac12 +\beta & \textnormal{if } \beta > \frac{1}{2} \\
\alpha = 0  & \textnormal{if } \beta = \frac{1}{2} \\ 
\end{array} .
\]
For $q\to 1$, they converge to the standard normal distribution, see~\cite{Sza01}.

\section{Exponentially Growing Parameter}\label{sec:schnellPar}

\begin{satz}\label{thm:conv}
 If $X_n \sim KB(n, \theta q^{-n} ,q)$, then for $n\to \infty$,
\[
 n-X_n \longrightarrow H \left( \frac{q}{\theta} \right)
\]

\end{satz}
\begin{proof}
 Define $Y_n=n-X_n$. Then, by (\ref{equ:1}) with $z= \theta q^{-n}$,
\begin{align*}
 \P (Y_n = x) &= \qbinom{n}{x}{q}\frac{ q^{-n(n-x)} q^{(n-x)(n-x-1)/2} \theta ^{n-x}}{\prod_{i=0}^{n-1} \left( 1+\theta \frac{q^i}{q^n} \right)}\\
 &= \qbinom{n}{x}{q} \frac{q^{-\frac{n^2+n-x^2-x}{2}}\theta ^{n-x}}{q^{n(n-1)/2} q^{-n^2}\theta ^n \prod_{i=0}^{n-1} (1+q^{n-i}\theta^{-1}) }\\
& =  \qbinom{n}{x}{q} \frac{q^{x(x+1)/2}\theta ^{-x}}{\prod_{i=1}^{n} (1+q^i\theta^{-1})}
= \qbinom{n}{x}{q} q^{x(x-1)/2} \left( \frac{q}{\theta}\right)^x \frac{1}{\left( -\frac{q}{\theta},q\right)_n} .
\end{align*}
Therefore
\[
 \P (Y_n=x) \to \frac{q^{x(x-1)/2}}{(q,q)_x} \left( \frac{q}{\theta}\right)^x e_q\left( -\frac{q}{\theta } \right) .
\]
\end{proof}

If the parameter grows only slightly faster than in Theorem~\ref{thm:conv}, then the limit
distribution is degenerate.

\begin{satz}
 If $X_n \sim KB(n, q^{-n-f(n)} ,q)$ with $f(n)\to \infty$ for $n\to \infty$, then for $n\to \infty$,
\[
 n-X_n \longrightarrow \delta_0,
\]
where $\delta_0$ denotes the point measure in $0$.
\end{satz}
\begin{proof}
 Define $Y_n=n-X_n$. Then
\begin{align*}
 \P (Y_n=x) &= \qbinom{n}{x}{q} \frac{ q^{(-n-f(n))(n-x)}q^{(n-x)(n-x-1)/2}}{\prod_{i=0}^{n-1} (1+q^{i-n-f(n)})} \\
 &= \qbinom{n}{x}{q} \frac{q^{-(n^2+2nf(n)+n-xf(n)-x^2-x)/2}}{\prod_{i=0}^{n-1} (1+q^{i-n-f(n)})}.
\end{align*}
It suffices to prove 
\[
 \lim_{n \to \infty}\frac{q^{-(n^2+2nf(n)+n)/2}}{\prod_{i=0}^{n-1} (1+q^{i-n-f(n)})} = 1,
\]
which is equivalent to
\[
 \lim_{n \to \infty}q^{(n^2+2nf(n)+n)/2}\prod_{i=0}^{n-1} (1+q^{i-n-f(n)}) = 1.
\]
By (\ref{equ:1})
\begin{align*}
 q^{(n^2+2nf(n)+n)/2}\prod_{i=0}^{n-1} (1+q^{i-n-f(n)}) 
   & =  \prod_{i=0}^{n-1} (1+q^{n+f(n)-i})\\
   & =  \prod_{i=0}^{n-1} (1+q^{f(n)+1+i}) .
\end{align*}
This tends to~$1$ as $n\to\infty$ by Lemma~\ref{lem:prod}.
\end{proof}

{\bf Acknowledgement.} We thank Stephan Wagner for the idea of using Mellin transform
asymptotics for the sequence of means in Section~\ref{sec:langPar}.

\bibliographystyle{siam}
\bibliography{literature}

\end{document}